\documentclass[12pt]{amsart}

\usepackage{amssymb, amsmath, amsthm,hyperref}

\newcommand{\Z}{\mathbb{Z}}

\newtheorem{theorem}{Theorem}
\newtheorem{proposition}{Proposition}
\newtheorem{lemma}{Lemma}

\begin{document}

\title{Bunching for relatively pinched metrics}

\author{Yannick Guedes Bonthonneau}

\date{August 2024}

\begin{abstract}
We consider compact Riemannian manifolds whose curvature tensor is pointwise negatively pinched, and improve on the corresponding unstable bunching estimate from Hasselblatt's 1994 paper. 
\end{abstract}

\maketitle

Let $(M,g)$ be a compact (strictly) negatively curved Riemannian manifold. Denote by $K$ its sectional curvature function. Traditionally, $g$ is said to be $a^2$ negatively pinched for some $a\in(0,1)$ if globally, 
\[
-1 \leq K \leq - a^2. 
\]
It is a well established result that if $g$ is $a^2$ negatively pinched, the horospherical foliation of $g$ has regularity $C^{2a-}$. This was proven by Boris Hasselblatt in full generality in \cite{Hasselblatt-94-1}, and then improved on in \cite{Hasselblatt-97}. We recommend the introduction of \cite{Hasselblatt-94-1} for a review of the many results on regularity of dynamical foliations.

Following \cite{Hasselblatt-94-2}, we will say that $g$ is $a^2$ relatively negatively pinched if for every $x\in M$, and every pair of planes $P,Q$ through $x$, 
\[
\frac{K(x,P)}{K(x,Q)} \geq a^2. 
\]
In \cite{Hasselblatt-94-2}, Hasselblatt proved that in this case, the horospherical foliation is $C^{2a^2-}$. We claim that actually
\begin{theorem}\label{thm:main}
Under relative negative $a^2$ pinching, for $a\in (0,1)$, the horospherical foliation is $C^{2a-}$. 
\end{theorem}
It is a fair question to ask whether it is possible at all for a $a^2$ relatively pinched metric to not be $a^2$ pinched (except in dimension $2$ where the relative pinching is always $1$). We provide an example in any dimension $\geq 3$ in \S \ref{sec:example-pinching}.

In \cite{Gerber-Hasselblatt-Keesing-2003}, Marlies Gerber, Boris Hasselblatt \& Daniel Keesing proved that the technical statement in \cite{Hasselblatt-94-2} was optimal in some sense, so let us explain shortly how this is possible. Using traditional comparison technique, one is lead to study pairs of solutions to the scalar Ricatti equations, in the form
\begin{align*}
\dot{x}_1 + x_1^2 &= k_1\\
\dot{x}_2 + x_2^2 &= k_2,
\end{align*}
where $k_2\geq a^2 k_1$. Assuming that $x_1$ and $x_2$ are defined and positive for $t\geq 0$, Hasselblatt essentially proves that 
\begin{equation}\label{eq:pointwise-ratio-a^2}
\frac{x_1}{x_2}(0)\geq a^2 \Rightarrow \frac{x_1}{x_2} \geq a^2, \forall t\geq 0. 
\end{equation}
If one could replace $a^2$ by $a$ in the RHS, one would obtain our theorem. However Gerber, Hasselblatt and Keesing proved that no such bound can hold in general. We will argue that actually, it is not the ratio \eqref{eq:pointwise-ratio-a^2} that is the crucial object, but
\begin{equation}\label{eq:average-ratio}
\liminf_{T\to+\infty}\frac{\int_0^T x_1}{\int_0^T x_2}
\end{equation}
It is a priori not clear how one would compute such a ratio without relying on the pointwise ratio in general. However, when estimating the Lyapunov exponent of closed orbits, we are led to Ricatti equations with periodic forcing $k_j$, so that we can compute these averages with a finite time computation. Density of closed orbits will then let us conclude, in the spirit of \cite{Hasselblatt-94-3}. It would be interesting to find a more direct argument. 

More precisely, our arguments prove that under $a^2$ relative pinching, for any periodic point $x$,
\[
\limsup_{t\to+\infty} \frac{1}{t}\log \| (d_x\varphi_t)_{|E^u}\| \| (d_x \varphi_t)_{|E^u}^{-1}\|^{\frac{1}{a}} \leq 0. 
\]
We will relatedly obtain that 
\[
\limsup_{t\to+\infty} \frac{1}{t}\log \sup_x \| (d_x\varphi_t)_{|E^u}\| \| (d_x \varphi_t)_{|E^u}^{-1}\|^{\frac{1}{a}} \leq 0. 
\]
Actually, for periodic points, inspecting our proof, we could probably prove for every periodic point $x$ that 
\[
\limsup_{t\to+\infty} \| (d_x\varphi_t)_{|E^u}\| \| (d_x \varphi_t)_{|E^u}^{-1}\|^{\frac{1}{a}} < \infty,
\]
with the $\limsup$ equal to $0$ if along the orbit of $x$, the maximum curvature $\|K(x)\|$ was not constant. If one were able to find a more direct approach and work with all orbits and not only periodic ones, maybe one could prove that for some constant $C>0$,  
\[
\limsup_{t\to +\infty}\sup_x \| (d_x\varphi_t)_{|E^u}\| \| (d_x \varphi_t)_{|E^u}^{-1}\|^{\frac{1}{a}} \leq C.
\]
In this case, the results from \cite{Hasselblatt-97} would apply. The horospherical foliation would be $C^{2a}$ when $a\neq 1/2$, and $C^{\mathcal{O}(x\log x)}$ when $a=1/2$. It would also be interesting then to determine whether the result of Ursula Hamenst\"adt \cite{Hamenstadt-93} could be extended, and prove that in the relative $1/4$ pinching case, one gets an invariant open and dense set of points where the foliation is $C^1$. 

To conclude this introduction, we mention that the regularity of the horospherical foliation is a crucial tool in the proof of many rigidity statements. For this reason, in several theorem of the litterature, it would be possible to replace ``$1/4$'' or ``$1/9$ pinched'' by $1/4$ or $1/9$ \emph{relatively} pinched. As a very partial selection of references we mention \cite{Butt-2023,Gogolev-RodriguezHertz-2024}

\textbf{Acknowledgment.} I am indebted to motivating discussions with Thibault Lefeuvre. I am also grateful to Sebastien Gou\"ezel's (2020!) explanations. Finally, I am supported by PRC grant ADYCT (ANR-20-CE40-0017).

\section{An integral inequation}

The main technical ingredient in our proof is the following inequality, which may (or may not be) of independent interest.
\begin{proposition}\label{prop:crucial-inequality}
Let $\mu>0$ be a smooth $1$-periodic function with average $1$. Then for $h>0$, 
\[
h (1 - e^{-1/h}) \leq \int_0^1 d\tau \mu(\tau)^2 \int_0^1 dt e^{-\frac{1}{h}\int_\tau^{\tau+t} \mu}. 
\]
\end{proposition}

\begin{proof}
Let us start by observing that if $\mu_n$ is a sequence of $L^2$ functions, positive with average $1$, and satisfying the inequality, tending to $\mu$ in $L^2$ norm, then $\mu$ must satisfy the inequality. On the other hand, if the Proposition is true, we can replace ``smooth'' by $L^2$ by a density argument. It thus suffices to prove the inequality for functions that are constant by parts. Additionally, by translation invariance, we can restrict our attention to a more particular class of functions. We take $C_j,\epsilon_j > 0$, $j=1\dots N$ numbers with
\[
\epsilon_1+ \dots + \epsilon_N = C_1 \epsilon_1 + \dots + C_N \epsilon_N = 1, 
\]
and require that additionally, for $j=1\dots N$,
\begin{equation}\label{eq:well-distributed}
C_j \epsilon_j = \frac{1}{N}.
\end{equation}
Now we consider the function $\mu$ that is constant by parts, with the successive values $C_1, \dots, C_N$, on intervals of length $\epsilon_1, \dots, \epsilon_N$. Let us denote by $I_j$ these successive intervals.

In this case, we can actually compute everything, and prove the inequality directly. By periodicity, we observe that once we have computed the contribution to the integral from values of $\tau\in I_1$, we will be able to deduce the contributions from the other intervals $I_j$, $j=2\dots N$.

For $\tau\in I_1$, and $j = 2 \dots N$,
\begin{align*}
\int_0^{\epsilon_1-\tau} dt e^{- \frac{1}{h} \int_\tau^{t+\tau}\mu} &= \frac{h}{C_1}(1- e^{\frac{C_1 (\tau-\epsilon_1)}{h} }), \\
\int_{I_j} dt 				e^{- \frac{1}{h} \int_\tau^{t+\tau}\mu} &= \frac{h}{C_j} e^{\frac{C_1 \tau}{h} - \frac{C_1\epsilon_1+\dots + C_{j-1}\epsilon_{j-1}}{h}}(1 - e^{- \frac{C_j\epsilon_j}{h}}),\\
\int_{1-\tau}^1 			e^{- \frac{1}{h} \int_\tau^{t+\tau}\mu} &= \frac{h}{C_1} e^{\frac{C_1 \tau}{h} - \frac{1}{h}}(1 - e^{-\frac{C_1\tau}{h}}).
\end{align*}
The contribution from $I_1$ to the RHS of the inequality is thus
\[
\begin{split}
h C_1 \epsilon_1 (1- e^{-1/h}) + &h^2 C_1 ( e^{\frac{C_1 \epsilon_1}{h}}- 1) \Big[ - \frac{1}{C_1}e^{-\frac{C_1\epsilon_1}{h}} + \frac{1}{C_1} e^{- \frac{1}{h}} \\
&\quad + \sum_{j=2}^N \frac{1}{C_j} (1- e^{-\frac{C_j\epsilon_j}{h}}) e^{- \frac{ C_1 \epsilon_1 + \dots + C_{j-1} \epsilon_{j-1}}{h}} \Big]
\end{split}
\]
We see that when we sum the contributions from all the intervals, the first term will sum to the left hand side of the inequality, so we only need to prove that the sum of the contribution from the $h^2$ term is nonnegative. Let us now use \eqref{eq:well-distributed}: the second part of the contribution above is 
\[
h^2 C_1( e^{\frac{1}{Nh}}- 1) \Big[ - \frac{1}{C_1}e^{-\frac{1}{Nh}} + \frac{1}{C_1} e^{- \frac{1}{h}}  + \sum_{j=2}^N \frac{1}{C_j} (1- e^{-\frac{1}{Nh}}) e^{- \frac{ j-1 }{N h}} \Big]
\]
Now, we can remove the $h^2 (e^{1/Nh}-1)$ term that is not essential here, and sum over the intervals $I_j$. Finally, it suffices to prove that
\[
0 \leq \sum_{k=1}^N \left( - e^{-\frac{1}{Nh}} + e^{-1/h} + \sum_{j=1}^{N-1} \frac{C_k}{C_{k+j}}(1-e^{-1/Nh})e^{- \frac{j}{Nh}}\right). 
\]
(here we extended the notation $C_j$ periodically for all $j\in\Z$). We introduce a telescopic sum:
\[
 e^{-1/h} - e^{-\frac{1}{Nh}} = - (1 - e^{-\frac{1}{Nh}})( e^{-\frac{1}{Nh}} + \dots + e^{-\frac{N-1}{Nh}}).
\]
The inequality becomes
\[
0 \leq (1- e^{-\frac{1}{Nh}}) \sum_{j=1}^{N-1} e^{-\frac{j}{Nh}}\sum_{k=1}^N \left( \frac{C_k}{C_{k+j}} - 1 \right). 
\]
We need a lemma
\begin{lemma}
for any numbers $a_j>0$, $j=1\dots M$, denoting $a_{M+1}=a_1$,
\[
\sum_{j=1}^M \frac{a_j}{a_{j+1}} \geq M,
\]
with equality if and only if all $a_j$'s are equal. 
\end{lemma}

\begin{proof}
We start by observing that for $a,b,c>0$,
\[
\frac{a}{b} + \frac{b}{c} \geq 2\sqrt{\frac{a}{c}} ,
\] 
with equality if and only if $b= \sqrt{ac}$. Next, we also have
\[
2\sqrt{\frac{a}{b}} + \frac{b}{c} \geq 3 \sqrt[3]{\frac{a}{c}},
\]
with equality if and only if $b = (c \sqrt{a})^{2/3}$. More generally,
\[
n \sqrt[n]{\frac{a}{b}} + \frac{b}{c} \geq (n+1) \sqrt[n+1]{\frac{a}{c}},
\]
with equality if and only if $b= (c a^{\frac{1}{n}})^{\frac{n}{n+1}}$. Proceeding by finite induction, we deduce that 
\[
\sum_{j=1}^M \frac{a_j}{a_{j+1}} \geq = M \sqrt[M]{\frac{a_1}{a_{M+1}}}= M. 
\]
Without loss of generality, we may assume that the $a_j$'s are ordered, $a_1\leq a_2 \leq \dots \leq a_M$. According to equality cases above, there can only be equality if (at least), 
\[
a_M = (a_{M+1} a_1^{\frac{1}{M}})^{\frac{M}{M+1}} = a_1. 
\]
On the other hand, if all $a_j$'s are equal, there certainly is equality. 
\end{proof}

Let us consider the sum ($j\geq 1$),
\[
\sum_{k=1}^N \frac{C_k}{C_{k+j}}-1.
\]
We can decompose the sum according to the orbits of $k \mapsto k+j$ in $\Z/N\Z$, and apply our lemma above to deduce that this sum is non-negative. If not all $C_k$'s are equal, then it must be positive at least when $j=1$. This completes the proof of the Proposition.
\end{proof}

Let us consider the quotient of the RHS by the LHS of the inequality in the Proposition. It is given by
\[
1 + h \frac{(1-e^{-\frac{1}{Nh}})^2}{1 - e^{-\frac{1}{h}}} \sum_{j=1}^{N-1} e^{-\frac{j-1}{Nh}}\sum_{k=1}^N \frac{C_k}{C_{k+j}} - 1 . 
\]
When $h$ is very small this is
\[
1 + h \sum_{k=1}^N \frac{C_k}{C_{k+1}} - 1 + \mathcal{O}(e^{-\frac{1}{Nh}}),
\]
so that the inequality becomes an equality in the limit $h\to0$. On the other hand the limit when $h\to + \infty$ is
\[
1 + \frac{1}{N^2}\sum_{j=1}^{N-1} \sum_{k=1}^N \frac{C_k}{C_{k+j}} - 1 ,
\]
which we can check to be 
\[
\int_0^1 \mu^2. 
\]
(and the inequality becomes asymptotically an instance of the Cauchy Schwarz inequality).

\section{Periodic Ricatti equation and strength of forcing}

We will now interpret the result of the previous section in terms of Ricatti equations. Let us take $f>0$ a smooth $1$-periodic function, and denote by $\lambda_a$ the unique smooth positive $1$-periodic solution to 
\[
\dot{\lambda} + \lambda^2 = a^2 f. 
\]
That such a solution is well defined follows from the usual monotonicity properties of the Ricatti equation. Additionally, it depends smoothly on the parameter $a$. Let us set ($a>0$)
\[
\Lambda(a)= \int_0^1 \lambda_a. 
\]
Again, from usual arguments, we know that the thus-defined function $\Lambda$ is increasing with $a$. We also observe that
\[
\int_0^1 \lambda_a^2 = a^2 \int_0^1 f,
\]
so that 
\begin{equation}\label{eq:upper-bound-from-Cauchy-Schwarz}
\Lambda(a) \leq a \sqrt{\int_0^1 f}. 
\end{equation}
We will prove
\begin{proposition}\label{prop:almost-convexity}
Under the standing assumptions, 
\[
a \mapsto \frac{\Lambda(a)}{a}
\]
is non-increasing. 
\end{proposition}
This suggests that the map $\Lambda$ is concave, but we will not determine whether it is the case or not. The rest of this section is devoted to the proof of the Proposition.
\begin{proof}
We start by changing slightly the form of the equation: changing parameter $h=1/a$, and considering
\[
\mu_h = h \lambda_{1/h}, 
\]
we need to prove that 
\[
g : h \mapsto \int_0^1 \mu_h
\]
is non-decreasing. The equation satisfied by $\mu_h$ is given by
\begin{equation}\label{eq:mu}
h \dot{\mu}_h + \mu_h^2 = f. 
\end{equation}
Heuristically, we expect that as $h\to 0$, $\mu_h \sim \sqrt{f}$, so that \eqref{eq:upper-bound-from-Cauchy-Schwarz} becomes the Cauchy-Schwarz inequality. On the other hand, when $h\to+\infty$, $\mu_h$ should become increasingly constant, and \eqref{eq:upper-bound-from-Cauchy-Schwarz} should become saturated. In any case, using the smoothness of the map $g$, let us compute its derivative. 

Certainly, denoting $\nu_h = \partial_h \mu_h$, 
\[
h \dot{\nu} + \dot{\mu} + 2 \mu \nu = 0. 
\]
Let us solve this equation. Starting with
\[
h \frac{d}{dt}(\nu e^{\frac{2}{h}\int_0^t \mu}) = - \dot{\mu}e^{\frac{2}{h}\int_0^t \mu},
\]
we get
\[
h \nu = C e^{- \frac{2}{h}\int_0^t \mu } - \int_0^t d\tau \dot{\mu}(\tau) e^{-\frac{2}{h}\int_{\tau}^t \mu }. 
\]
Here we can integrate by parts:
\begin{align*}
- \int_0^t d\tau \dot{\mu}(\tau) e^{-\frac{2}{h} \int_\tau^t \mu} &= [ - \mu e^{-\frac{2}{h}\int_{\tau}^t \mu} ]_0^t + \frac{2}{h}\int_0^t d\tau \mu(\tau)^2 e^{-\frac{2}{h}\int_\tau^t \mu}, \\
				&= \mu(0) e^{-\frac{2}{h} \int_0^t \mu} - \mu(t)  + \frac{2}{h} \int_0^t d\tau \mu(\tau)^2 e^{-\frac{2}{h}\int_{\tau}^t \mu}. 
\end{align*}
We can thus write 
\[
h\nu(t) = C e^{- \frac{2}{h}\int_0^t \mu } - \mu(t) + \frac{2}{h} \int_0^t d\tau \mu(\tau)^2 e^{-\frac{2}{h}\int_{\tau}^t \mu }. 
\]
The function $\nu$ must be $1$ periodic, so that
\[
C = C e^{- \frac{2}{h}\int_0^1 \mu} + \frac{2}{h} \int_0^1 d\tau {\mu}(\tau)^2 e^{-\frac{2}{h} \int_\tau^1 \mu}. 
\]
We thus get the formula
\[
\begin{split}
\frac{h^2}{2} \int_0^1 \nu = - \frac{h}{2}\int_0^1 \mu &+ \int_0^1 dt \int_0^t d\tau \mu(\tau)^2 e^{- \frac{2}{h}\int_\tau^t \mu} \\
			&+ \frac{\int_0^1 dt e^{-\frac{2}{h} \int_0^t \mu} \int_0^1 d\tau \mu(\tau)^2 e^{-\frac{2}{h}\int_\tau^1 \mu} }{1- e^{-\frac{2}{h}\int_0^1 \mu}}
\end{split}
\]
We will be done if we can prove that $\int \nu \geq 0$, i.e prove that
\[
\frac{h}{2}\int_0^1 \mu \leq \int_0^1 dt \int_0^t d\tau \mu(\tau)^2 e^{- \frac{2}{h}\int_\tau^t \mu} +  \frac{\int_0^1 dt e^{-\frac{2}{h} \int_0^t \mu} \int_0^1 d\tau \mu(\tau)^2 e^{-\frac{2}{h}\int_\tau^1 \mu} }{1- e^{-\frac{2}{h}\int_0^1 \mu}}
\]
We can simplify this a little. First, forgetting that $\mu$ solves \eqref{eq:mu}, we observe that if the inequality holds for all values of $h>0$ for some function $\mu$, then it also holds for all values of $h>0$ for any $\alpha \mu$, $\alpha>0$. We can thus assume that $\int_0^1 \mu =1$. We can also replace $h$ by $2h$ everywhere, and we are led to the inequality
\[
h \leq \int_0^1 dt \int_0^t d\tau \mu(\tau)^2 e^{- \frac{1}{h}\int_\tau^t \mu} +  \frac{\int_0^1 dt e^{-\frac{1}{h} \int_0^t \mu} \int_0^1 d\tau \mu(\tau)^2 e^{-\frac{1}{h}\int_\tau^1 \mu} }{1- e^{-1/h}}
\]
We can rearrange this a little bit further. Multiplying by $(1-e^{-1/h})$, the RHS becomes
\[
\int_0^1 dt \int_0^t d\tau \mu(\tau)^2 e^{- \frac{1}{h}\int_\tau^t \mu}(1-e^{-1/h}) +  \int_0^1 dt e^{-\frac{1}{h} \int_0^t \mu} \int_0^1 d\tau \mu(\tau)^2 e^{-\frac{1}{h}\int_\tau^1 \mu} 
\]
This is
\[
\int_0^1 d\tau \mu(\tau)^2\left[ \int_0^1 dt e^{-\frac{1}{h}(\int_\tau^t \mu + 1)} + \int_\tau^1 dt e^{-\frac{1}{h}\int_\tau^t \mu} - \int_\tau^1 e^{-\frac{1}{h}(\int_\tau^t + 1)} \right]
\]
We recognize
\[
\int_0^1 d\tau \mu(\tau)^2 \int_0^1 dt e^{-\frac{1}{h}\int_\tau^{\tau+t} \mu}. 
\]
Finally, we can apply Proposition \ref{prop:crucial-inequality} from the previous section, and this closes the proof. 
\end{proof}

The main consequence of Proposition \ref{prop:almost-convexity} is that for $0< a <1$,
\begin{equation}\label{eq:good-bunching-ricatti}
\Lambda(a) \leq \Lambda(1) \leq \frac{\Lambda(a)}{a}. 
\end{equation}

\section{Proof of main theorem}

Let us now come back to our main problem. Let us start with the key estimate
\begin{proposition}
Let $\gamma$ be a periodic orbit of the geodesic flow of $(M,g)$ a relatively negatively $a^2$ pinched compact manifold. Let $\lambda_-^u$ (resp. $\lambda_+^u$) be the smallest (resp. largest) unstable lyapunov exponent of $\gamma$. Then
\[
0 < \lambda_-^u \leq \lambda_+^u \leq \frac{\lambda_-^u}{a}. 
\]
\end{proposition}

\begin{proof}
Lyapunov exponents of the periodic orbit are related to the asymptotic growth of unstable Jacobi fields, whose basic description we recall now. For more detail, see \cite{Klingenberg-book}. Denote by $T$ the length of the orbit $\gamma$. Orthogonal Jacobi fields can be described (using some parallel transport) as matrix solutions to
\[
J'' + K J = 0, 
\]
where $K(t)$ is the value of the curvature matrix along $\gamma$, which is symmetric. The unstable solution is the unique solution tending to $0$ in negative time, and is denoted $\mathbb{J}^u$. In particular, 
\[
\lambda_-^u = \liminf_{t\to+\infty} \frac{1}{t} \log \| \mathbb{J}^u(t)^{-1} \|^{-1},
\]
and
\[
\lambda_+^u = \limsup_{t\to+\infty} \frac{1}{t} \log \| \mathbb{J}^u(t)\|
\]
Let us now introduce the unstable Ricatti matrix:
\[
\mathbb{U}(t)=  \mathbb{J}' \mathbb{J}^{-1}. 
\]
This is a symmetric matrix, let us denote by $\lambda_-(t)$ (resp. $\lambda_+(t)$) its smallest (resp. largest) eigenvalue. Then for $v\neq 0$,
\[
\frac{\|\mathbb{J}v\|'}{\|\mathbb{J}v\|} = \frac{\langle\mathbb{J}'v, \mathbb{J}v\rangle}{\|\mathbb{J}v\|^2}\in [\lambda_-(t), \lambda_+(t)].
\]
We deduce by Gr\"onwall's inequality that 
\[
\|\mathbb{J}v\| \in \|v\| [ e^{\int_0^t \lambda_-}, e^{\int_0^t \lambda_+} ]. 
\]
However the unstable Ricatti matrix is $T$-periodic and positive, so that we find
\[
\frac{1}{T}\int_0^T \lambda_-(t) dt \leq \lambda_-^u \leq \lambda_+^u \leq \frac{1}{T} \int_0^T \lambda_+(t)dt. 
\]
The proof will thus be complete if we can prove that 
\[
\int_0^T \lambda_+(t)dt \leq \frac{1}{a} \int_0^T \lambda_-(t)dt.
\]
The unstable Ricatti matrix satisfies
\begin{equation}\label{eq:matrix-Ricatti}
\dot{\mathbb{U}} + \mathbb{U}^2 = - K. 
\end{equation}
Because of eigenvalue crossings, the functions $\lambda_\pm$ may be non-smooth. To avoid as much as possible this kind of complications, we observe that if we approximate the periodic smooth function $-K$ uniformly by trigonometric polynomials, and solve the corresponding Ricatti equation, we will smoothly approximate $\mathbb{U}$. In particular for our purposes, we can assume from now on that $-K$ is a positive, matrix-valued, trigonometric polynomial. In this case, we can further assume that except at a finite number of points, the eigenvalues of $\mathbb{U}$ and $K$ are simple. Let us call $\Omega\subset [0,1]$ the complement of these points. In $\Omega$, we can find smooth $v_\pm(t)$ unit vectors so that 
\[
\mathbb{U}v_\pm = \lambda_\pm v_\pm. 
\]
Additionally, we can assume that the $v_\pm$ have left and right limits at each point of $\Omega^c$. Then
\begin{align*}
\dot{\lambda}_\pm 	&= \frac{d}{dt}\left[ \langle \mathbb{U} v_\pm, v_\pm \rangle \right] \\
					&= \langle (-\mathbb{U}^2 - K) v_\pm, v_\pm \rangle,\\
					&= - \lambda_\pm^2 - \langle K v_\pm, v_\pm \rangle. 
\end{align*}
We see that the $\dot{\lambda}_\pm$ are continuous by parts, with well-defined limits at the points of $\Omega^c$. Let us denote 
\[
f(t) = \max_{\|v\|=1} \langle -K v, v\rangle. 
\]
Again because of eigenvalue crossings, $f$ is continuous, but only smooth on $\Omega$. By assumption, for $\|v\|=1$,
\[
a^2 f(t) \leq \langle K v, v \rangle \leq f(t)
\]
From this we see that on $\Omega$
\begin{align*}
\dot{\lambda}_+ + \lambda_+^2 &\leq f(t),\\
\dot{\lambda}_- + \lambda_-^2 &\leq a^2 f(t).
\end{align*}
Let us denote by ${\eta}_\pm$ the periodic positive solutions to 
\begin{align*}
\dot{\eta}_+ + \eta_+^2 &= f(t),\\
\dot{\eta}_- + \eta_-^2 &= a^2 f(t).
\end{align*}
We must have
\[
\eta_- \leq \lambda_- \leq \lambda_+ \leq \eta_+. 
\]
To see this, we observe that on $\Omega$
\[
\dot{\lambda}_+ - \dot{\eta}_+ \leq \eta_+^2 - \lambda_+^2 = (\eta_+ - \lambda_+)(\eta_+ + \lambda_+). 
\]
This inequality also applies to the left and right limits of $\dot{\lambda}_+$ at the points of $\Omega^c$. We deduce that extremal points of $\lambda_+ - \eta_+$, 
\[
\eta_+^2 \geq \lambda_+^2,
\]
so that globally $\eta_+ \geq \lambda_+$. A similar argument applies to $\eta_-$ and $\lambda_-$. 

Applying again an approximation argument, we can now without loss of generality assume that $f$ is smooth. We can rescale the time to obtain $1$-periodic instead of $T$-periodic functions and then apply Proposition \ref{prop:almost-convexity}, and in particular Inequality \eqref{eq:good-bunching-ricatti} to close the proof
\end{proof}

Let us recall Theorem 5 from \cite{Hasselblatt-94-2}:
\begin{theorem}\label{thm:Hasselblatt-94}
Assume that for some $\alpha\in (0,2)$,
\[
\lim_{t\to +\infty} \sup_{x\in SM} \| (d_x\varphi_{-t})^{-1}_{|E^u}\| \| (d_x\varphi_{-t})_{|E^u}\|^{\frac{2}{\alpha}} = 0.
\]
Then the horospherical foliation is $C^\alpha$. 
\end{theorem}

As we have seen above, we can estimate this quantity using Jacobi and Ricatti matrix fields. More precisely, the condition in the theorem is equivalent to requiring the vanishing as $t\to +\infty$ of the supremum over $x\in SM$ of
\[
\| \mathbb{J}^u(t) \| \| (\mathbb{J}^u(t))^{-1} \|^{\frac{2}{\alpha}},
\]
where $\mathbb{J}^u$ is the unstable matrix Jacobi field based at $x$. We have also seen (using some sub-additivity) that if $\mathbb{U}$ is the unstable Ricatti matrix based at $x$, and $\lambda_\pm(x,t)$ its extremal eigenvalues, the quantity above is controlled by 
\[
\exp \left( \int_0^t \lambda_+ - \frac{2\lambda_-}{\alpha} \right). 
\]
In particular, the conclusion of Theorem \ref{thm:Hasselblatt-94} holds if we can prove that 
\[
\lim_{t\to +\infty} \frac{1}{t} \sup_{x\in SM} \int_0^t \lambda_+(x,t) - \frac{2\lambda_-(x,t)}{\alpha} < 0. 
\]
Now, the functions $\lambda_\pm(x,t)$ satisfy 
\[
\lambda_\pm(x,t) = \lambda_\pm(\varphi_t(x),0), 
\]
so we can rewrite this as 
\[
\max\overline{\zeta}:=\lim_{t\to +\infty} \sup_{x\in SM} \frac{1}{t}\int_0^t \zeta(\varphi_t(x)) < 0, 
\]
for
\[
\zeta(x) = \lambda_+(x,0) - \frac{2\lambda_-(x,0)}{\alpha}. 
\]
Let us now use a classical trick (Thank you Sébastien Gou\"ezel !!). Let $x_n$ be a sequence of points for which the supremum at $t=n$ is attained, and consider the probability measure $\mu_n$ along the length $n$ orbit of $x_n$. We can extract to ensure that $\mu_n$ converges weakly to some invariant probability measure $\mu$. Since $\zeta$ is a continuous function, we then have 
\[
\int \zeta d\mu_n \to \int \zeta d\mu = \max\overline{\zeta}. 
\]
Using the decomposition of invariant measures into ergodic component, and the Birkhoff theorem we deduce that 
\[
\max\overline{\zeta} = \sup_{\mu} \int \zeta d\mu,
\]
where the supremum is taken over all invariant probability measures. Finally, it is well known that the Anosov Closing Lemma implies that every invariant probability measure is the weak limit of probability measures supported on periodic orbits. We thus conclude that whenever $(M,g)$ is compact, relatively negatively $a^2$ pinched, and $\alpha < 2a$,
\[
\max\overline{\zeta} < 0. 
\]
Combining this with Theorem \ref{thm:Hasselblatt-94}, we obtain a proof of Theorem \ref{thm:main}. 

\section{Coming back to the example of Gerber, Hasselblatt and Keesing}

As mentionned in the introduction, direct estimation of the ratio between solutions to the Ricatti equations, instead of taking ratio of averages is bound to fail, according to Gerber, Hasselblatt and Keesing. Let explain this phenomenon. 

The example proposed in \cite{Gerber-Hasselblatt-Keesing-2003} is not exactly the same as the one we describe, but the fundamentals are similar. Let us set
\[
f(t) = \begin{cases} 1,& t\in [0,1/2]+\Z \\ \epsilon^2,& t\in]1/2,1[ + \Z \end{cases},
\]
where $0<\epsilon<1$ is small enough. Recalling that $\tanh$ and $\coth$ are the non-constant solutions to the Ricatti equation with constant forcing $1$, we deduce that if $\lambda_a$ is the positive periodic solution to ($a>0$)
\[
\dot{\lambda}_a + \lambda_a^2 = a^2f,
\]
Then
\[
\lambda_a(t)= \begin{cases} a\tanh(a(t -1/2) + t_0) & t\in[0,1/2] \\  a\epsilon \coth( a\epsilon(t-1/2) + t_1)& t\in[1/2,1] \end{cases} 
\]
for some $t_0,t_1$. The continuity condition gives
\begin{align*}
\tanh(t_0) & = \epsilon \coth t_1 \\
\tanh(t_0 - a/2)&= \epsilon \coth( t_1 + a/2). 
\end{align*}
Let us now consider the situation where $\epsilon$ is small, and $a$ becomes very large. Then since $t_0> a/2$, it must also be very large, and we get $\coth t_1 \sim 1/\epsilon$. We also get that $\tanh(t_0- a/2) \sim \epsilon$. This gives
\[
t_1 \sim \epsilon,\quad t_0 - a/2 \sim \epsilon. 
\]
Let us now consider $a_1 < a_2$ both very large, so that $a_2 = a_1/\epsilon$ and take $t= \epsilon/a_1$. Then
\[
\lambda_{a_1}(t) = a_1 \tanh(a_1 t + t_0(a_1) - a_1/2) \sim 2 a_1 \epsilon,
\]
while
\[
\lambda_{a_2}(t) = a_2 \tanh(a_2 t + t_0(a_2) - a_2/2) \sim a_2 = \frac{a_1}{\epsilon}. 
\]
It cometh
\[
\frac{\lambda_{a_1}(t)}{\lambda_{a_2}(t)} \sim 2 \epsilon^2 = 2 \left(\frac{a_1}{a_2}\right)^2.
\]

However, let us now compute averages, in the same scaling limit. Actually, for most of the interval $[0,1/2]$, $\lambda_{a_j}\sim a_j$, and for most of the interval $[1/2,1]$, $\lambda_{a_2} \sim a_2\epsilon$, and we always have $\lambda_{a_1}\geq a_1 \epsilon$. The ratio of averages thus satisfies the expected
\[
\frac{\overline{\lambda_{a_1}}}{ \overline{\lambda_{a_2}}} \gtrsim \frac{ a_1/2 }{a_2 / 2}. 
\]

\section{A metric that is better relatively pinched than pinched}
\label{sec:example-pinching}

\begin{proposition}
In any dimension $n\geq 3$, for any compact $n$-hyperbolic manifold $(M,g)$, there exists a metric $\tilde{g}$ on $M$, $C^3$ close to $g$, such that if $a_1$ is the best local pinching constant, and $a_2$ is the best global pinching constant, 
\[
0< a_1 < 1,\quad \frac{1- a_2}{1- a_1} \geq \frac{3}{2}. 
\]
In particular, $\tilde{g}$ is (strictly) better relatively pinched than pinched.
\end{proposition}

\begin{proof}

We will construct a perturbation localized in a very small ball in hyperbolic space, so that every computation is local. In radial coordinates $(r,\omega)$, recall that the expression for the hyperbolic metric is
\[
g=dr^2 + \sinh(r)^2 g_{\mathbb{S}^{n-1}}(d\omega)
\]
Let us consider a metric of the form $\tilde{g} = e^{2\varphi} g$, with
\[
\varphi = \epsilon \chi\left( \frac{r - r_0 }{\epsilon^\alpha} \right).
\]
We require that $\chi\in C^\infty_c(]-1,1[)$ is not the $0$ function, and $r_0 > \epsilon^\alpha >0$ are small, and $\alpha\in (0,1/3)$. The new metric $\tilde{g}$ coincides with the hyperbolic metric near $r=0$, and for $r> 2 r_0$. It is $C^3$ close to $g$ since $\alpha < 1/3$. Recall that the sectional curvature of the plane generated by orthogonal $X,Y$ for $\tilde{g}$ is given by 
\[
\tilde{K}(X,Y) = e^{-2 \varphi}\left[ - 1  - \nabla^2_{X,X} \varphi - \nabla^2_{Y,Y}\varphi - |\nabla \varphi|^2 + (X \varphi)^2 + (Y \varphi)^2 \right].
\]
(every Riemannian quantity in the RHS computed with respect to the metric $g$). Here, 
\[
- |\nabla \varphi|^2 + (X \varphi)^2 + (Y \varphi)^2 = \mathcal{O}(\epsilon^{2 - 2 \alpha}). 
\]
On the other hand,
\begin{align*}
\nabla^2_{X,X}\varphi 	&= \langle X, \nabla_X \nabla \varphi \rangle \\
						&= \epsilon^{1-\alpha}\langle X, \nabla_X( \chi'((r-r_0)/\epsilon^\alpha) \nabla r) \rangle \\
						&= \epsilon^{1-2\alpha} \langle X, X(r) \chi''((r-r_0)/\epsilon^\alpha) \nabla r \rangle + \mathcal{O}(\epsilon^{1-\alpha})\\	
						&= \epsilon^{1-2\alpha} |X(r)|^2 \chi''\left( \frac{r - r_0}{\epsilon^\alpha} \right) + \mathcal{O}( \epsilon^{1-\alpha} ). 
\end{align*}
Since $\varphi=\mathcal{O}(\epsilon)$, we deduce that 
\[
\tilde{K}(X,Y) =  -1 - \epsilon^{1-2\alpha} (|X(r)|^2 + |Y(r)|^2) \chi''\left( \frac{r - r_0}{\epsilon^\alpha} \right) + \mathcal{O}(\epsilon^{1-\alpha}) 
\]
Since we have $\partial_r r = 1$, and $\partial_\omega r = 0$, $|X(r)|^2 + |Y(r)|^2$ takes all values between $0$ and $1$ when we sample all the planes at a point. In particular, we deduce that $\tilde{g}$ is $a_1^2$ relative pinched with 
\[
a_1 = \min_{r} 1 - \frac{\epsilon^{1-2\alpha}}{2}\left|\chi''\left( \frac{r - r_0}{\epsilon^\alpha} \right)\right| + \mathcal{O}(\epsilon^{1-\alpha}). 
\]
That is
\[
a_1 = 1 - \frac{\epsilon^{1-2\alpha}}{2} \max |\chi''| + \mathcal{O}(\epsilon^{1-\alpha}). 
\]
Let us now consider global pinching. Since $\chi$ does not identically vanish, and is compactly supported, $\chi''$ must crucially take both positive and negative signs. The computation above shows that $\tilde{g}$ cannot be better globally pinched than 
\[
1 - \frac{\epsilon^{1-2\alpha}}{2}( \max \chi'' - \min \chi'' ) + \mathcal{O}(\epsilon^{1-\alpha}). 
\]
We can always arrange so that $\min\chi'' = - \max \chi'' < 0$, so that we find
\[
\frac{1- a_2 }{1- a_1} \geq 2 + \mathcal{O}(\epsilon^{\alpha}). 
\]
Taking $\epsilon$ small enough ensures that the RHS is larger than $3/2$, so that
\[
\frac{1- a_2 }{1- a_1} \geq \frac{3}{2}, 
\]
and
\[
a_2 \leq 1 - \frac{3}{2}(1-a_1) = \frac{3}{2}a_1 - \frac{1}{2} \leq a_1 - \frac{1-a_1}{2} < a_1. 
\]
\end{proof}

%\bibliographystyle{alpha}
%\bibliography{biblio}

\begin{thebibliography}{GHK03}

\bibitem[But23]{Butt-2023}
Karen Butt.
\newblock Approximate rigidity of the marked length spectrum.
\newblock {\em Math. Res. Rep. (Amst.)}, 4:63--82, 2023.

\bibitem[GHK03]{Gerber-Hasselblatt-Keesing-2003}
Marlies Gerber, Boris Hasselblatt, and Daniel Keesing.
\newblock The {Riccati} equation: pinching of forcing and solutions.
\newblock {\em Exp. Math.}, 12(2):129--134, 2003.

\bibitem[GRH24]{Gogolev-RodriguezHertz-2024}
Andrey Gogolev and Federico Rodriguez~Hertz.
\newblock Smooth rigidity for higher-dimensional contact {Anosov} flows.
\newblock {\em Ukr. Math. J.}, 75(9):1361--1370, 2024.

\bibitem[Ham91]{Hamenstadt-91}
Ursula Hamenst{\"a}dt.
\newblock Compact manifolds with 1/4-pinched negative curvature.
\newblock Global differential geometry and global analysis, {Proc}. {Conf}.,
  {Berlin}/{Ger}. 1990, {Lect}. {Notes} {Math}. 1481, 73-78 (1991)., 1991.

\bibitem[Ham93]{Hamenstadt-93}
Ursula Hamenst{\"a}dt.
\newblock Regularity of time-preserving conjugacies for contact {Anosov} flows
  with {{\(C^ 1\)}}-{Anosov} splitting.
\newblock {\em Ergodic Theory Dyn. Syst.}, 13(1):65--72, 1993.

\bibitem[Has94a]{Hasselblatt-94-1}
Boris Hasselblatt.
\newblock Regularity of the {Anosov} splitting and of horospheric foliations.
\newblock {\em Ergodic Theory Dyn. Syst.}, 14(4):645--666, 1994.

\bibitem[Has94b]{Hasselblatt-94-2}
Boris Hasselblatt.
\newblock Horospheric foliations and relative pinching.
\newblock {\em J. Differ. Geom.}, 39(1):57--63, 1994.

\bibitem[Has94c]{Hasselblatt-94-3}
Boris Hasselblatt.
\newblock Periodic bunching and invariant foliations.
\newblock {\em Math. Res. Lett.}, 1(5):597--600, 1994.

\bibitem[Has97]{Hasselblatt-97}
Boris Hasselblatt.
\newblock Regularity of the {Anosov} splitting. {II}.
\newblock {\em Ergodic Theory Dyn. Syst.}, 17(1):169--172, 1997.

\bibitem[Kli95]{Klingenberg-book}
Wilhelm P.~A. Klingenberg.
\newblock {\em Riemannian geometry.}, volume~1 of {\em De Gruyter Stud. Math.}
\newblock Berlin: Walter de Gruyter, 2nd ed. edition, 1995.

\end{thebibliography}

\end{document}